\documentclass[10pt]{amsart}
\usepackage{enumerate,amsmath,amssymb,latexsym,
amsfonts, amsthm, amscd, MnSymbol}

\theoremstyle{plain}

\newtheorem{theorem}{Theorem}

\numberwithin{equation}{section}

\addtocounter{section}{-1}

\begin{document}

\title {Implicational Propositional Calculus\\Tableaux and Completeness}

\date{}

\author[P.L. Robinson]{P.L. Robinson}

\address{Department of Mathematics \\ University of Florida \\ Gainesville FL 32611  USA }

\email[]{paulr@ufl.edu}

 \keywords{}

\begin{abstract}

We discuss tableaux for the Implicational Propositional Calculus and show how they may be used to establish its completeness. 

\end{abstract}

\maketitle

\medbreak

\section{Introduction} 

Tableaux [5] and dual tableaux [1] facilitate not only an intuitive analysis of the logical structure of propositional formulas but also an elegant proof that the classical Propositional Calculus is complete in the sense that all tautologies are theorems. The Implicational Propositional Calculus (IPC) has the conditional ($\supset$) as its only primitive connective and modus ponens (MP) as its only inference rule; it is traditionally based on the three axiom schemes 
$$ [(A \supset B) \supset A] \supset A$$
$$ A \supset (B \supset A)$$
$$[A \supset (B \supset C)] \supset [(A \supset B) \supset (A \supset C)]$$
of which the first scheme is due to Peirce. It is well-known that IPC is complete: a Kalm\`ar approach is outlined in Exercises 6.3 - 6.5 of [2]; a Lindenbaum approach is offered in [3]. Our purpose in this paper is to adapt the signed dual tableaux of [6] and thereby present another proof of IPC completeness.  

\section{IPC: Tableaux and Completeness}

\medbreak 

Throughout, we shall make free use of the fact that the Deduction Theorem (DT) holds in IPC: if $A$ and the set $\Gamma$ of (well-formed IPC) formulas furnish a deduction of $B$, then $\Gamma$ alone furnishes a deduction of $A \supset B$; with standard symbolism, if $\Gamma, A \vdash B$ then $\Gamma \vdash A \supset B$. A special case is usually proved en route to DT: namely, $A \supset A$ is a theorem scheme of IPC. A particularly useful consequence of DT is Hypothetical Syllogism (HS): $A \supset B, B \supset C \vdash A \supset C$. None of these results requires the Peirce axiom scheme. 

\medbreak 

We shall also make quite full use of the following result, which is Exercise 6.3 in [2]. Here, $QA = Q(A)$ stands for $A \supset Q$ so that $QQA = (A \supset Q) \supset Q$ and so on. 

\medbreak 

\begin{theorem} \label{Robbin} 
Each of the following is a theorem scheme of IPC: \par
{\rm (1)} $(A \supset B) \supset [(B \supset C) \supset (A \supset C)]$ \par 
{\rm (2)} $(A \supset B) \supset (QB \supset QA)$ \par 
{\rm (3)} $A \supset QQA$ \par
{\rm (4)} $QQQA \supset QA$ \par 
{\rm (5)} $QQB \supset QQ(A \supset B)$ \par 
{\rm (6)} $QQA \supset [QB \supset Q(A \supset B)]$ \par 
{\rm (7)} $QA \supset QQ(A \supset B)$ \par 
{\rm (8)} $(QA \supset B) \supset [(QQA \supset B) \supset QQB].$ 
\end{theorem} 

\begin{proof} 
Omitted; this is an exercise! We note that only (7) requires the Peirce axiom scheme; indeed, (7) implies the Peirce scheme when $Q: = A$. 
\end{proof} 

\medbreak 

Disjunction may be defined within IPC: explicitly, 
$$A \vee B := (A \supset B) \supset B.$$
Not only is it the case that $A \vdash A \vee B$ and $B \vdash A \vee B$: by virtue of the Peirce axiom scheme, it is also the case that if $A \vdash C$ and $B \vdash C$ then $A \vee B \vdash C$; for this and related observations, see [4]. Observe at once that $B \vee A \vdash A \vee B$ and vice versa: thus, if $B \vee A$ is a theorem of IPC then so is $A \vee B$ and conversely; we shall make use of this observation in what follows. 

\medbreak 

We may define multiple disjunctions inductively, thus 
$$A_1 \vee \dots \vee A_{N + 1} = (A_1 \vee \dots \vee A_N) \vee A_{N + 1}.$$

\medbreak 

\begin{theorem} \label{dis}
{\rm (1)} If $1 \leqslant n \leqslant N$ then $A_n \vdash A_1 \vee \cdots \vee A_N.$\\
{\rm (2)} When $1 \leqslant n \leqslant N$ let $A_n \vdash B$; then $A_1 \vee \cdots \vee A_N \vdash B.$\\
{\rm (3)} When $1 \leqslant n \leqslant N$ let $A_n \vdash B_n$; then $A_1 \vee \cdots \vee A_N \vdash B_1 \vee \cdots \vee B_N.$
\end{theorem} 

\begin{proof} 
(1) and (2) fall to elementary induction arguments. For (3): if $1 \leqslant n \leqslant N$ then $B_n \vdash B_1 \vee \cdots \vee B_N$ by (1) so that $A_n \vdash B_1 \vee \cdots \vee B_N$ by hypothesis and (2) ends the argument. 
\end{proof} 

\medbreak 

We record an immediate application of this theorem: as $QQ A = (A \supset Q) \supset Q = A \vee Q$ it follows that 
$$QQ (A_1 \vee \cdots \vee A_N) \vdash (QQ A_1) \vee \cdots \vee (QQ A_N).$$

\medbreak

We remark that although conjunctions themselves may not be definable within IPC, their proxies may appear as the antecedents to conditionals: in many respects, the IPC formula $A \supset (B \supset C)$ behaves as if it were $(A \wedge B)  \supset C$; of course, this reflects classical  importation and exportation. 

\medbreak 

We shall assume as known the theory of signed dual tableaux, for details of which we refer to [6]; see especially Chapter 6 and Chapter 7. Signed formulas of Type A in IPC have the form $\alpha = F (X \supset Y)$ with $\alpha_0 = T X$ and $\alpha_1 = F Y$; signed formulas of Type B in IPC have the form $\beta = T (X \supset Y)$ with $\beta_0 = F X$ and $\beta_1 = T Y$. According to the branching rules for signed dual tableaux, $\alpha$ has $\alpha_0$ and $\alpha_1$ as alternative consequences while $\beta$ has $\beta_0$ and $\beta_1$ as direct consequences: symbolically, 
$$\frac{\alpha}{\alpha_0 \; | \; \alpha_1} \; \; \; \; \; \frac{\beta}{\beta_0} \; .$$
$$\; \; \; \; \; \; \; \; \; \; \; \; \; \; \; \;  \beta_1$$

\medbreak

If $Z$ is a tautology of IPC then $T Z$ starts a complete signed dual tableau that is closed in the sense that each of its branches contains a signed formula and its conjugate (where conjugation is the interchange of $T$ and $F$). Here, the fact that we work within IPC rather than within the classical Propositional Calculus is essentially immaterial: the branching consequences of IPC formulas are themselves IPC formulas, and the conditional has the same semantic content in both systems. 

\medbreak 

Now, let us choose and fix a (well-formed) formula $Q$ of IPC; recall the abbreviations $Q W := W \supset Q$ and $QQ W := (W \supset Q) \supset Q$. Let $Z$ be a tautology of IPC (written without making use of the aforementioned abbreviations) and construct a closed signed dual tableau $\mathcal{T}$ starting from $T Z$. To $\mathcal{T}$ we apply the following transformation: replace each node of the form $T W$ by $QQW$; replace each node of the form $F W$ by $QW$. The result of this transformation is a tableau $\mathcal{T}_Q$ starting from $QQ Z$ with the following branching rules: \\

 A:  $\; \; \alpha = Q (X \supset Y)$ has alternative consequences $\alpha_0 = QQ X$ and $\alpha_1 = Q Y$;\\

 B:  $\; \; \beta = QQ (X \supset Y)$ has direct consequences $\beta_0 = Q X$ and $\beta_1 = QQ Y$. \\

\noindent
Each complete branch $\theta$ of $\mathcal{T}_Q$ is a sequence $(Z_0, \dots, Z_N)$ where $Z_0 = QQ Z$ and the form of each term $Z_n$ is either $Q W_n$ or $QQ W_n$ for some formula $W_n$; to say that each branch of $\mathcal{T}$ is closed is to say that each such $\theta$ contains among its terms a conjugate pair $Q W$ and $QQ W$ for some formula $W$. 

\medbreak 

In our presentation, we have taken advantage of the theory of signed dual tableaux. Instead, we could start afresh: define a dual $Q$-tableau for $Z$ to be a dual tableau starting from $QQ Z$ and constructed according to the branching rules A and B displayed above. A $Q$-tableau for $Z$ would start from $Q Z$ and follow branching rules according to which an $\alpha$ has direct consequences and a $\beta$ has alternative consequences. 

\medbreak 

Following the pattern established in [6] we now introduce a uniform axiom system $\mathbb{U}_Q$ with axiom schemes and inference rules defined as follows. Here, $\theta = (Z_0, \dots , Z_N)$ is a sequence in which each term $Z_n$ has the form $Q W_n$ or $QQ W_n$; this sequence is closed precisely when some pair of terms has the form $Z_i = Q W$ and $Z_j = QQ W$. \\

{\it Axioms}: All disjunctions $D (\theta) := Z_0 \vee \cdots \vee Z_N$ for which $\theta = (Z_0, \dots \, Z_N)$ is closed. \\

{\it Rule A}: If $\alpha$ is a term of $\theta$ then from $ D(\theta) \vee \alpha_0$ and $D(\theta) \vee \alpha_1$ (together) infer $D(\theta)$. \\

{\it Rule B}: If $\beta$ is a term of $\theta$ then from $D(\theta) \vee \beta_0$ or $D(\theta) \vee \beta_1$ (separately) infer $D(\theta)$. \\ 

\medbreak 

The symbol $\vdash$ will continue to signify deducibility within IPC. It will now be convenient to write $\mathbb{T}$ for the set of theorems of IPC. 

\medbreak 

\begin{theorem} \label{axioms}
Each axiom of $\mathbb{U}_Q$ is a theorem of IPC. 
\end{theorem} 

\begin{proof} 
Abbreviate the axiom $D (\theta) = Z_0 \vee \cdots \vee Z_N$ of $\mathbb{U}_Q$ by $D$ for convenience. Say $\theta$ contains the conjugate terms $Z_i = Q W$ and $Z_j = QQ W$. Each of $QW \supset D$ and $QQ W \supset D$ lies in $\mathbb{T}$ by Theorem \ref{dis}. Theorem \ref{Robbin} part (8) tells us that $(Q W \supset D) \supset [(QQ W \supset D) \supset QQ D] \in \mathbb{T}$. By two applications of MP we deduce that $QQ D \in \mathbb{T}$. Now, $QQ (Z_0 \vee \cdots \vee Z_n) \supset (QQ Z_0 \vee \cdots \vee QQ Z_N) \in \mathbb{T}$ by DT and the formula recorded after Theorem \ref{dis}; also, $ (QQ Z_0 \vee \cdots \vee QQ Z_N) \supset (Z_0 \vee \cdots \vee Z_N) \in \mathbb{T}$ by Theorem \ref{Robbin} part (4) and Theorem \ref{dis}. Finally, two further applications of MP place $D$ in $\mathbb{T}$ as claimed. 
\end{proof} 

\medbreak 

Rule A of $\mathbb{U}_Q$ is effectively a derived inference rule for IPC. 

\medbreak

\begin{theorem} \label{A}
Let $\theta$ have $\alpha$ as a term. If $ D(\theta) \vee \alpha_0 \in \mathbb{T}$ and $D(\theta) \vee \alpha_1 \in \mathbb{T}$ then $D(\theta) \in \mathbb{T}$. 
\end{theorem} 

\begin{proof} 
Again write $D = D(\theta)$ for convenience. Say $\alpha = Q(X \supset Y)$ so that $\alpha_0 = QQ X$ and $\alpha_1 = Q Y$. As $\alpha$ is a term of $\theta$, it follows by Theorem \ref{dis} that $\alpha \supset D \in \mathbb{T}$. Theorem \ref{Robbin} part (6) tells us that $\alpha_0 \supset (\alpha_1 \supset \alpha) \in \mathbb{T}$ whence $\alpha_0, \alpha_1 \vdash \alpha$ by MP twice. A further application of MP now yields $\alpha_0, \alpha_1 \vdash D$ whence DT twice yields $\alpha_0 \supset (\alpha_1 \supset D) \in \mathbb{T}$. An application of HS to this and the hypothesis $(\alpha_1 \supset D) \supset D = \alpha_1 \vee D \in \mathbb{T}$ yields $\alpha_0 \supset D \in \mathbb{T}$; an application of MP to this and the hypothesis $(\alpha_0 \supset D) \supset D = \alpha_0 \vee D \in \mathbb{T}$ places $D$ in $\mathbb{T}$. 
\end{proof} 

\medbreak 

Rule B of $\mathbb{U}_Q$  may likewise be viewed as a derived inference rule for IPC. 

\medbreak

\begin{theorem} \label{B}
Let $\theta$ have $\beta$ as a term. If $D(\theta) \vee \beta_0 \in \mathbb{T}$ or $D(\theta) \vee \beta_1 \in \mathbb{T}$ then $D(\theta) \in \mathbb{T}$. 
\end{theorem} 

\begin{proof} 
Yet again write $D = D(\theta)$ for convenience. Say $\beta = QQ(X \supset Y)$ so that $\beta_0 = Q X$ and $\beta_1 = QQ Y$. \par 

Assume that $\beta_0 \vee D \in \mathbb{T}$. As $\beta$ is a term of $\theta$, we have $\beta \supset D \in \mathbb{T}$ by Theorem \ref{dis}. Theorem \ref{Robbin} part (7) tells us that $\beta_0 \supset \beta \in \mathbb{T}$. An application of HS yields $\beta_0 \supset D \in \mathbb{T}$; an application of MP to this and the assumption $(\beta_0 \supset D) \supset D \in \mathbb{T}$ places $D$ in $\mathbb{T}$. \par 

Assume that $\beta_1 \vee D \in \mathbb{T}$. As $\beta$ is a term of $\theta$, we have $\beta \supset D \in \mathbb{T}$ by Theorem \ref{dis}. Theorem \ref{Robbin} part (5) tells us that $\beta_1 \supset \beta \in \mathbb{T}$. An application of HS yields $\beta_1 \supset D \in \mathbb{T}$; an application of MP to this and the assumption $(\beta_1 \supset D) \supset D \in \mathbb{T}$ places $D$ in $\mathbb{T}$.
\end{proof} 

\medbreak 

These three theorems come together as follows. 

\medbreak 

\begin{theorem} \label{theorem}
Each theorem of $\mathbb{U}_Q$ is a theorem of IPC. 
\end{theorem} 

\begin{proof} 
The axioms of $\mathbb{U}_Q$ are themselves theorems of IPC by Theorem \ref{axioms}; Theorem \ref{A} and Theorem \ref{B} show that the set of theorems of IPC is closed under Rule A and Rule B.  
\end{proof} 

\medbreak 

The following result is complementary. 

\medbreak 

\begin{theorem} \label{tautology}
If the IPC formula $Z$ is a tautology then $QQ Z$ is a theorem of $\mathbb{U}_Q$. 
\end{theorem} 

\begin{proof} 
Construct a closed signed dual tableau $\mathcal{T}$ for $TZ$ and then the transformed tableau $\mathcal{T}_Q$ starting from $QQ Z$ as above. Each branch $\theta$ of $\mathcal{T}_Q$ is closed, so each corresponding disjunction $D(\theta)$ is an axiom of $\mathbb{U}_Q$. Prune the tableau by reversing the steps in its construction: at each stage, pruning a branch $\theta$ applies Rule A or Rule B to the corresponding disjunction $D(\theta)$ and therefore yields a theorem of $\mathbb{U}_Q$; the final pruning exposes the root $QQ Z$ so we conclude that $QQ Z = D(QQZ)$ is a theorem of $\mathbb{U}_Q$. 
\end{proof} 

We are now in a position to present a new proof of completeness for IPC. 

\medbreak 

\begin{theorem} 
The Implicational Propositional Calculus IPC is complete. 
\end{theorem} 

\begin{proof} 
Let $Z$ be a tautology of IPC. Theorem \ref{tautology} tells us that if $Q$ is any IPC formula then $QQ Z$ is a theorem of $\mathbb{U}_Q$ and Theorem \ref{theorem} tells us that $QQ Z = (Z \supset Q) \supset Q$ is a theorem of IPC. Now simply take $Q := Z$!  This shows that $(Z \supset Z) \supset Z$ is a theorem of IPC; MP and the Peirce axiom scheme then show that $Z$ itself is a theorem of IPC. 
\end{proof}

\bigbreak

\begin{center} 
{\small R}{\footnotesize EFERENCES}
\end{center} 
\medbreak 

[1] H. Leblanc and D.P. Snyder, {\it Duals of Smullyan Trees}, Notre Dame Journal of Formal Logic (1972) XIII(3) 387-393. 

[2] J. W. Robbin, {\it Mathematical Logic - A First Course}, W.A. Benjamin (1969); Dover Publications (2006).

[3] P.L. Robinson, {\it Implicational Completeness}, arXiv 1511.02953 (2015). 

[4] P.L. Robinson, {\it The Peirce axiom scheme and suprema}, arXiv 1511.07074 (2015). 

[5] R.M. Smullyan, {\it Logical Labyrinths}, A.K. Peters (2009). 

[6] R.M. Smullyan, {\it A Beginner's Guide to Mathematical Logic}, Dover Publications (2014). 

\medbreak

\end{document}